\documentclass[12pt]{amsart}
\usepackage{amsmath,amssymb,amsthm,mathtools}
\usepackage{fourier}
\usepackage{tikz}

\theoremstyle{plain}
\newtheorem{theorem}{Theorem}
\newtheorem{lemma}[theorem]{Lemma}
\newtheorem*{theorem*}{Theorem}

\begin{document}
\author{Iwan Praton}
\author{Weiran Zeng}
\address{Franklin \& Marshall College} \email{ipraton@fandm.edu, wzeng@fandm.edu}
\title{Amicable Heronian Parallelograms}
\date{}
\maketitle
\begin{abstract}
A convex polygon is \emph{Heronian} if its side lengths and its area are integers. Two polygons are \emph{amicable} if the area of one is equal to the perimeter of the other, and vice versa. We show that there are infinitely many pairs of amicable Heronian parallelograms, and we give necessary and sufficient conditions for a Heronian parallelogram to be part of an amicable pair. 
\end{abstract}

A convex polygon is said to be \emph{Heronian} if its side lengths are integers and its area is also an integer. (Note that we do not require the diagonals to be integers.) Heronian triangles have received a lot of attention; see \cite{Ca}, \cite{Y}, and \cite{PS}, for example. Heronian quadrilaterals have not been investigated quite as thoroughly, although articles about them exist, e.g., \cite{BM}. We aim to add to the collection with this paper.

Two polygons are \emph{amicable} if the area of one is the perimeter of the other, and vice versa. With a slight abuse of terminology, each polygon in an amicable pair is also called amicable. In \cite{PS} it was shown that there is only one pair of amicable Heronian triangles. In \cite{PZ} it was shown there are five pairs of Heronian rectangles. The finiteness seems to end with this case. In this follow-up note, we show that there are infinitely many amicable Heronian parallelograms, and we also provide a characterization of these amicable parallelograms. 

Triangles and rectangles are of course completely determined by their side lengths. For parallelograms, the two side lengths are not enough to determine the parallelogram. We designate one of the sides as the \emph{base} of the parallelogram, and the other side, not parallel to the base, as its \emph{side}. The \emph{height} of the parallelogram is the distance between the base and the side parallel to it. The base, side, and height do determine the parallelogram.

\begin{lemma}{\label{lemma1}}
Let $b$ and $s$ be positive integers, and let $A$ be any positive integer with $A\leq bs$. Then there is a Heronian parallelogram with base $b$, side $s$, and area $A$. 

Similarly, suppose $b$ and $h$ are positive integers, and let $u$ be any positive integer with $u\geq h$. Then there is a Heronian parallelogram with base $b$, height $h$, and side $u$. 
\end{lemma}
\begin{proof}
This is a continuity argument, the same as Proposition 1 in \cite{AC1}. We start with a rectangle with base $b$ and height $s$, then start flattening it while maintaining the side lengths, as shown in Figure 1. We get parallelograms (with base $b$ and side $s$) whose height continuously  decreases from $s$ to $0$. Since $A\leq bs$, there is a parallelogram whose height is $A/b$. Its area is thus $b(A/b)=A$, and this is the required parallelogram. 

A similar argument establishes the second part of the Lemma. See Figure 2. 
\end{proof}

\begin{figure}[h]
\quad \begin{tikzpicture}[scale=0.8]
\draw (0,0) rectangle (3,5);
\draw[dashed] (5,0) arc (0:90:5);
\draw (0,0)--(4,3)--(7,3)--(3,0);
\draw (0,0)--(4.9,1)--(7.9,1)--(3,0);
\node at (1.5,-0.2) {$b$};
\node at (-0.2,2.5) {$s$};
\end{tikzpicture}
\caption{Three parallelograms with the same sides}
\end{figure}
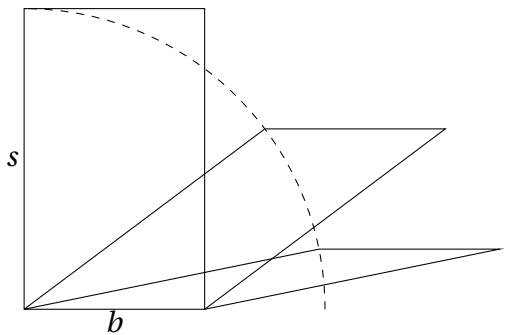

\bigskip

\begin{figure}[h]
\quad \begin{tikzpicture}[scale=0.8]
\draw (0,0) rectangle (3,5);
\draw[dashed] (0,5) -- (12,5);
\draw (0,0) -- (3.5,5) -- (6.5,5) -- (3,0);
\draw (0,0) -- (7.5,5) -- (10.5,5) -- (3,0);
\node at (1.5,-0.2) {$b$};
\node at (-0.2,2.5) {$h$};
\end{tikzpicture}
\caption{Three parallelograms with the same base and height}
\end{figure}
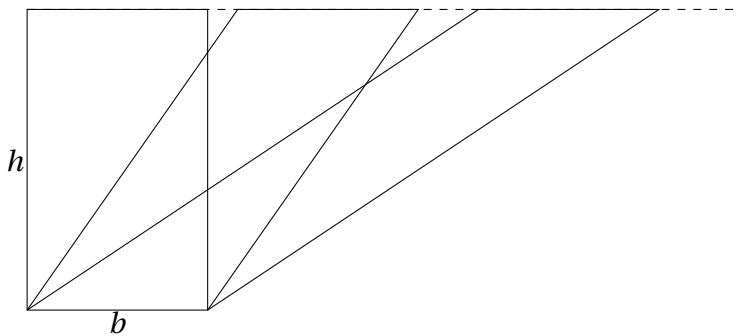

Lemma~\ref{lemma1} shows that we have quite a bit of flexibility in constructing Heronian paralllelograms with a given area or a given perimeter. Thus it should not be too hard to construct an amicable companion for a given Heronian parallelogram. In fact, it is relatively straightforward to construct an infinite family of amicable Heronian parallelograms. We present here an example of such a family where Fibonacci numbers make an appearance.

Let $F_n$  denote the $n$th Fibonacci number and $L_n$ the $n$th Lucas number. Consider the rectangle $H_n$ with base $L_n$ and height $2F_n$. We show that $H_n$ is amicable for $n>3$. Its companion is the parallelogram $C_n$ with base $F_{2n-2}$, side $F_{2n-1}$, and height $2F_{n+3}/F_{2n-2}$. 

We first check that $H_n$ and $C_n$ form an amicable pair. The area of $H_n$ is $2F_nL_n=2F_{2n}$. The perimeter of $C_n$ is $2(F_{2n-2} + F_{2n-1})=2F_{2n}$, the correct value. The perimeter of $H_n$ is $4F_n+2L_n= 4F_n+2F_{n-1}+2F_{n+1}= 2(F_{n-1}+F_n+F_n+F_{n+1}) = 2F_{n+3}$, which is indeed equal to the area of $C_n$.  

We also need to check that $C_n$ exists. By Lemma~\ref{lemma1} it suffices to check that $2F_{n+3} \leq F_{2n-1}F_{2n-2}$. For $n>3$, we have $2n-1\geq n+3$, so $F_{2n-1}\geq F_{n+3}$. And clearly $F_{2n-2}>2$ for $n>3$. Thus $F_{2n-1}F_{2n-2} \geq F_{n+3}$, as required. 

So there are lots of amicable Heronian parallelograms; they are not even that hard to find. But there do exist Heronian parallelogram that are not amicable.  For example, all Heronian parallelograms have even perimeter, hence amicable Heronian parallelograms must have even area. So Heronian parallelograms that have odd area are not amicable. Unfortunately, having even area is not enough. Theorem~\ref{thm} below, along with Lemma~\ref{lemma1}, imply that for any positive integer $N$, there exists a Heronian parallelogram with area $N$ that is not amicable. Similarly, there exists a non-amicable Heronian parallelogram with any given perimeter.

So we need a way to differentiate amicable Heronian parallelograms from those that are not. That is the content of our main theorem. 

\begin{theorem}{\label{thm}}
Let $H$ be a Heronian parallelogram with area $A$ and perimeter $P$.  Then $H$ is amicable if and only if $A$ is even and $(A/4)^2 \geq P$.
\end{theorem}
\begin{proof}
Denote the base of $H$ by $a$, its height by $h$, and its side by $s$. Then $A=ha$ and $P=2a+2s$. Suppose $H$ has an amicable companion parallelogram $C$. Let $b$ denote the base of $C$, $t$ its height, and $u$ its side. Then we have $ha=2b+2u$ (which already implies that $A=ah$ is even) and $2a+2s=bt$. Thus
\[
t = \frac{2a+2s}{b}=\frac{P}{b}, \quad u=\frac{ha}{2} -b.
\]
Since $C$ is a parallelogram, its side must be at least as big as its height, so we have $u\geq t$. So
\[
\frac{P}{b} \geq \frac{ha}{2}-b\iff P\geq \frac12 A b -b^2\iff
b^2 -\frac{A}{2}b+P\leq 0.
\]
The parabola $x^2-(A/2)x+P$ has a minimum at $x=A/4$ and its minimum value is $(A/4)^2 - (A/2)(A/4) + P = P - (A/4)^2$. This minimum value must be $\leq 0$, so we have $(A/4)^2 \geq P$. This establish the necessity.

To show sufficiency, suppose $A$ is even and $(A/4)^2\geq P$. We show that there is an integer $b$ such that $b^2 -(A/2)b+P\leq 0$. If $A/4$ is an integer, then we simply use $b=A/4$. If $A/4$ is not an integer, then it is a half integer, so $A/4 +1/2$ is an integer. We also know that $(A/4)^2 > P$ because the left side is not an integer but the right side is. This implies that $A^2/4 > 4P$; since both sides are integers, we have $A^2/4 \geq 4P+1$. Now set $b=A/4+1/2$, an integer. Then
\[
b^2-\frac{A}{2}b+P = \left(\frac{A}4+\frac12\right)^2 - \frac{A}{2}\left(\frac{A}4+\frac12\right) +P = -\frac{A^2}{16} +\frac14 +P = \frac{4P+1-A^2/4}{4} \leq 0,
\]
as required.

We now set $t=P/b$ and $u=A/2 - b$. Note that $u$ is an integer and $u\geq t$, so we can construct a Heronian parallelogram $C$ with base $b$, height $t$, and side $u$. It is then easy to check that $C$ is an amicable companion to $H$. 
\end{proof}

One interesting consequence of this theorem is that the amicability of a  Heronian parallelogram depends only on its area and perimeter, and not on its sides.

\end{document}